\documentclass[12pt]{amsart}
\usepackage{lipsum}
\usepackage{amssymb}
\usepackage{enumerate}
\usepackage{amsthm}
\usepackage{tikz-cd}
\usepackage{amsmath}
\usepackage[mathscr]{euscript}
\usepackage[utf8]{inputenc}
\usepackage[english]{babel}
\usepackage{hyperref}
\hypersetup{
    colorlinks=true,
    linkcolor=blue,
    citecolor=red,
    filecolor=red,
    urlcolor=violet,
}

\makeatletter
\@namedef{subjclassname@2020}{%
  \textup{2020} Mathematics Subject Classification}
\makeatother
\newtheorem{thm}{Theorem}[section]

\newtheorem{prop}[thm]{Proposition}
\newtheorem{corl}[thm]{Corollary}

\theoremstyle{definition}
\newtheorem{defi}[thm]{Definition}

\newtheorem{xrem}{Remark}

\DeclareMathOperator{\rank}{{rank}}

\DeclareMathOperator{\HN}{{HN}}
\DeclareMathOperator{\gr}{{gr}}

\begin{document}
\baselineskip=16pt

\author{Indranil Biswas}
\author{Snehajit Misra}
\author{Nabanita Ray}
\address{Department of Mathematics, Shiv Nadar University, NH91, Tehsil Dadri, Greater Noida,
Uttar Pradesh 201314, India}
\email[Indranil Biswas]{indranil.biswas@snu.edu.in}
\address{Department of Mathematics and Computing, Indian Institute of Technology (Indian School of Mines) Dhanbad, Jharkhand - 826004, India}
\email[Snehajit Misra]{snehajitm@iitism.ac.in}
\address{Indraprastha Institute of Information Technology, Delhi, Okhla Industrial Estate, Phase III,
New Delhi, Delhi 110020, India}
\email[Nabanita Ray]{nabanita@iiitd.ac.in}

\address{}
\email[]{}

\begin{abstract}
In \cite{BCO25}, Bruzzo, Capasso and Otero extended
the notion of ampleness of vector bundles to the more
general context of Higgs bundles. But the ampleness of Higgs bundles did not
coincide with the ampleness of vector bundles when the Higgs field is zero. We modify
the definition of ample Higgs bundles that results in removal of this discrepancy.
Invoking this definition, we study various properties Higgs ample vector bundles. In particular, we prove a Barton-Kleimann type criterion to characterize the Higgs ample vector bundles.
\end{abstract}

\title{ Positivity of Higgs Vector Bundles}
\maketitle

\vskip 4mm

\section{Introduction}
Over the last few decades, a number of notions of positivity of line bundles have been introduced to understand the geometry of the base variety. The notions of ampleness  and nefness of a line bundle play an important role in this understanding. The notion of ampleness of a line bundle was extended to higher rank vector bundles by Hartshorne in \cite{H66}. Later, the notion of nefness for higher rank vector bundles was introduced by Campana and Peternell in  \cite{CP91}. The so called Barton-Kleimann criteria were given in \cite{L2} to characterize ample vector bundles and nef vector bundles.

Subsequently, the notion of ampleness and nefness were defined in the set up of Higgs vector bundles by Bruzzo, Capasso and Otero
in \cite{BBG19} and \cite{BCO25}. It turns out that the notion of ampleness as introduced in \cite{BCO25} does not coincide with the usual notion of ampleness for ordinary vector bundles when the Higgs field is set to be zero. In view of this, we have introduced a modified notion of Higgs ampleness that actually generalizes the notion of ampleness for ordinary vector bundles when the Higgs field is zero.

Our definition of Higgs ampleness is weaker in the sense that every Higgs bundle which is Higgs ample in the sense of \cite{BCO25} is Higgs ample according to our definition. Moreover, it can happen that a Higgs bundle which is Higgs ample according to our definition is not Higgs ample according to \cite{BCO25}. Using the new definition, we have proved various properties of Higgs ample vector bundles. In particular, we proved a Barton-Kleimann type criterion to characterize the Higgs ample bundles. More specifically, we prove the following
(see Corollary \ref{corl3.4}):

\begin{thm}[{Barton-Kleimann type criterion for Higgs ampleness}] Let $X$ be a smooth projective variety and let $h \in \text{NS}(X)_{\mathbb{R}}$ be a fixed ample class on $X$. Then a Higgs vector bundle $\mathcal{E}=(E,\theta)$ is Higgs ample if and only if the  following two conditions holds :

    \vspace{2mm}

    \begin{itemize}
        \item The line bundle $\det(E)$ is ample.
        \item  There exists a $\delta\in \mathbb{R}_{>0}$ such that for every morphism $f:C\longrightarrow X$ from a smooth projective curve $C$, and for every Higgs quotient $\mathcal{Q}$ of $f^*\mathcal{E},$ the inequality
    $$\deg(\mathcal{Q})\geq \delta\bigl(C\cdot f^*h \bigr)$$ holds.
    \end{itemize}
    \end{thm}

Our proofs are inspired by the results in \cite{L1} and \cite{L2}. 

\section{Preliminaries}

Throughout this article, the base field  $\mathbb{K}$ is an algebraically closed field of characteristics 0. Let $X$ be a smooth projective variety over $\mathbb{K}$. A Higgs sheaf on $X$ is a pair $\mathcal{E}\,=\, (E,\,\theta)$, where $E$ is a coherent sheaf on $X$ and $\theta \,:\, E\,\longrightarrow \,E\otimes \Omega^1_X$ is an $\mathcal{O}_X$-module morphism such that the composition of maps
$$\theta \wedge \theta\,:\, E \, \xrightarrow{\,\,\theta\,\,}\, E \otimes \Omega^1_X\,\xrightarrow{\,\,\theta\otimes id\,\,}\, E\otimes \Omega^1_X \otimes \Omega_X^1\, \longrightarrow\, E\otimes \Omega^2_X $$
is the zero homomorphism. A Higgs sheaf $\mathcal{E}\, =\, (E,\,\theta)$ is called Higgs bundle if $E$ is locally free. For a smoth map $\phi : Y \longrightarrow X$ between two smooth projective varieties $X$ and $Y$, and a Higgs bundle $\mathcal{E}\,=\,(E,\,\theta)$ on $X$, its pullback $\phi^*\mathcal{E}$ under $\phi$ is defined to be the Higgs bundle
 $(\phi^*E,\,\phi^*\theta)$, where $\phi^*\theta\, :\, \phi^*E\, \longrightarrow\, \phi^*E\otimes \Omega^1_Y$ is the composition of maps
 \begin{align*}
\phi^*E\ \longrightarrow\ \phi^*E \otimes \phi^*\Omega^1_X\ \longrightarrow\ \phi^*E\otimes \Omega^1_Y.
 \end{align*}
A morphism of Higgs sheaves $f\,:\,{\mathcal E}\, :=\, (E,\,\theta)\,\longrightarrow\,
{\mathcal G}\, :=\, (G,\,\phi)$ is a morphism of $\mathcal{O}_X$-modules $f\,:\,E\,\longrightarrow\, G$ such that the following diagram is commutative:
\begin{center}
 \begin{tikzcd}
E \arrow[r, "f"] \arrow[d, "\theta"]
& G \arrow[d, "\phi" ] \\
E \otimes\Omega^1_X \arrow[r, "f\otimes id" ]
& G\otimes\Omega^1_X
\end{tikzcd}
\end{center}

The Higgs sheaves $\mathcal{E}$ and $\mathcal{G}$ are said to be isomorphic if there is a morphism $f \,:\, \mathcal{E} \,\longrightarrow\, \mathcal{G}$ such that $f\, :\, E\, \longrightarrow\, G$ is an isomorphism of $\mathcal{O}_X$-modules.

\subsection{Higgs Grassmann scheme}

Take a Higgs bundle $\mathcal{E} \,=\, (E,\,\theta)$ on $X$ of rank $r$, and also take an integer $1\leq k \leq r-1$. Let $p_k: \text{Gr}_k(E) \longrightarrow X$ be the Grassmann bundle parametrizing locally free quotients of $E$ of rank $k$. We consider the following exact sequence:
\begin{align}\label{seq1}
 0\, \longrightarrow\, S_{E,r-k}\, \xrightarrow{\,\,\,\psi\,\,\,}\, p_k^*E \, \xrightarrow{\,\,\,\eta\,\,\,}\, Q_{E,k}\,\longrightarrow\, 0,
\end{align}
where  $S_{E,r-k}$ is the universal rank $r-k$ subbundle of $p_k^*E$ and $Q_{E,k}$ is the universal quotient of rank $k$. Corresponding to a Higgs bundle $\mathcal{E}\,=\,(E,\,\theta)$, the closed subscheme $\mathcal{G}r_k(\mathcal{E}) \subseteq \text{Gr}_k(E)$ (called {\it Higgs Grassmann schemes}) is defined as the zero loci of the following composition of morphisms:
$$\bigl(\eta \circ \text{Id}_{\Omega^1_X}\bigr) \circ p_k^*(\theta)\circ \psi \ :\ S_{E,r-k}\ \longrightarrow\ Q_{E,k} \otimes p_k^*\Omega^1_X.$$
Restricting the exact sequence in \eqref{seq1} to the Higgs Grassmann scheme $\mathcal{G}r_k(\mathcal{E})$, we have the following universal short exact sequence
\begin{align}\label{seq2}
 0\ \longrightarrow\ \mathcal{S}_{\mathcal{E},r-k}\ \xrightarrow{\,\,\,\psi\,\,\,}\ \rho_k^*\mathcal{E}\ \xrightarrow{\,\,\,\eta\,\,\,}\
 \mathcal{Q}_{\mathcal{E},k}\ \longrightarrow\ 0.
\end{align}
Here $\rho_k \,=\, p_k\vert_{\mathcal{G}r_k(\mathcal{E})}$, and the Higgs vector bundle $\mathcal{Q}_{\mathcal{E},k} \,=\, Q_{E,k}\vert_{\mathcal{G}r_k(\mathcal{E})}$ is equipped with the quotient Higgs field induced by $\rho_k^*\theta.$ The Higgs Grassmann scheme 
$\mathcal{G}r_k(\mathcal{E})$ has the following universal properties :
\vspace{2mm}

    For any morphism $f\,:\,Y\,\longrightarrow\, X$ and any locally free Higgs quotient $\mathcal{F}$ of $f^*\mathcal{E}$ of rank $k$, there is  a morphism $\psi_k\,:\,Y\,\longrightarrow\, \mathcal{G}r_k(\mathcal{E})$ over $X$ such that $\mathcal{F} \,=\, \psi_k^*\mathcal{Q}_{\mathcal{E},k}.$
    Conversely, if such a map $\psi_k\,:\,Y\,\longrightarrow\, \mathcal{G}r_k(\mathcal{E})$ exists, then there is a map $f \,:\, Y\,\longrightarrow \,X$ such that $\mathcal{F}\, :=\, \psi_k^*\mathcal{Q}_{\mathcal{E},k}$ a locally free Higgs quotient of $f^*\mathcal{E}.$ 
\vspace{3mm}

For a morphism of varieties $f\,:\,Y\,\longrightarrow\, X,$ the morphism $\psi_k\,:\, Y \,\longrightarrow \,\text{Gr}_k(E)$ factors through the Higgs Grassmann scheme $\mathcal{G}r_k(\mathcal{E})$ if and only if $\theta$ induces a Higgs field on  the universal quotient bundle $\text{Q}_{E,k}$
over $\text{Gr}_{k}(E)$.
\vspace{2mm}

For each $1\,\leq\, k \,\leq\, r-1,$ we denote the line bundle $\mathcal{O}_{\text{Gr}_k(E)}(1)\vert_{\mathcal{G}r_k(\mathcal{E})}$ by $\mathcal{O}_{\mathcal{G}r(\mathcal{E})}(1).$

\subsection{Semi-stability of Higgs vector bundles}

Take a Higgs vector bundle $\mathcal{E} \,=\, (E,\,\theta)$ on a smooth projective curve $C$. Its slope $\mu(\mathcal{E})$ is defined as $$\mu(\mathcal{E})\ :=\ \mu(E) \ =\ \dfrac{\deg(E)}{\rank(E)}.$$

\begin{defi}\label{defn1}
 A  Higgs bundle $\mathcal{E} \,=\, (E,\,\theta)$ on a smooth projective curve $C$ is said to be {\it slope semistable} (respectively, {\it slope stable}) if  for every $\theta$-invariant proper subbundle $G$ of $E$ (i.e., $\theta(G) \,\subset\, G\otimes \Omega_C^1)$, one has
 \begin{center}
  $\mu(G)\ \leq\ \mu(E)$\ \ \, (respectively, $\mu(G)\ <\ \mu(E))$.
 \end{center}

\end{defi}  
  In fact,  in the Definition \ref{defn1} of semistability (respectively, stability) for a Higgs bundle $\mathcal{E}\,=\,(E,\,\theta)$, it is enough to consider $\theta$-invariant proper subbundles $G$ of $E$ for which the quotient $E/G$ is torsion-free (and hence locally free on the smooth curve $C$). 

  For a  Higgs bundle $\mathcal{E}\, =\, (E,\,\theta)$ on a smooth projective curve $C$, there is a unique filtration of  $\theta$-invariant subbundles
 \begin{align*}
  0 \,=\, \HN_0(\mathcal{E}) \,\subsetneq\, \HN_1(\mathcal{E}) \,\subsetneq\, \HN_2(\mathcal{E}) \,\subsetneq \,\cdots\,\subsetneq \,\HN_{l-1}(\mathcal{E})\, \subsetneq\, \HN_l(\mathcal{E}) \,=\, \mathcal{E}
 \end{align*}
such that each quotient Higgs bundle $\gr^{\HN}_i(\mathcal{E}) \,=\, \bigl((\HN_i(\mathcal{E})/\HN_{i-1}(\mathcal{E})),\overline{\theta}\vert_{\HN_{i-1}(\mathcal{E})}\bigr)$ is slope semistable with
$\mu(\gr^{\HN}_i(\mathcal{E})) \,>\, \mu(\gr^{\HN}_{i+1}(\mathcal{E}))$ for all $i\,=\,1,\,\cdots,\,l-1$. Such a filtration is called the Harder-Narasimhan filtration of $\mathcal{E}$.
We define $$\mu_{\max}(\mathcal{E}) \ :=\ \mu(\gr^{\HN}_1(\mathcal{E}))\hspace{3mm} \text{ and } \hspace{2mm} \mu_{\min}(\mathcal{E})\ :=\ \mu(\gr^{\HN}_l(\mathcal{E})).$$
Note that $$\mu_{\max}(\mathcal{E})\ \geq\ \mu(\mathcal{E})\ \geq\ \mu_{\min}(\mathcal{E})$$ and any equality holds if and only if $\mathcal{E}$ is slope semistable. Note that $\mu_{\min}(\mathcal{E})$ is the minimum value among all the slopes of all the non-zero Higgs quotient of $\mathcal{E}$.

\begin{defi}
\rm A Higgs bundle $\mathcal{E}$ is curve semistable if, for every morphism $f\,:\,C\,\longrightarrow\, X$ where $C$ is a smooth irreducible projective curve, the pullback bundle $f^*\mathcal{E}$ is semistable.
\end{defi}

\subsection{Higgs nef and Higgs ample bundles}

An ordinary vector bundle $E$ on a  projective scheme  $X$  is called ample (respectively, nef) if the tautological line bundle $\mathcal{O}_{\mathbb{P}(E)}(1)$ is ample (respectively, nef).
We now recall the defintions of Higgs nef and Higgs ample bundles.

\begin{defi}[{\cite{BCO25}}]\label{defi-nef}
Let $\mathcal{E}\, =\, (E,\,\theta)$ be a Higgs vector bundle of rank $r$ on a smooth  projective variety $X$. If $\rank(E) \,=\, 1$, then $\mathcal{E}$ is called {\it Higgs nef} ($\text{H}$-nef in short) if $E$ is nef as a line bundle in the usual sense. If $\rank(E) \,\geq \,2$, then the {\it Higgs nefness} ($\text{H}$-nefness in short) of $\mathcal{E}$ is defined inductively as follows:
\begin{enumerate}
    \item For all $k$ with $1\,\leq\, k \,\leq\, r-1$, the universal Higgs quotient bundles $\mathcal{Q}_{\mathcal{E},k}$ on Higgs Grassmann schemes $\mathcal{G}r_k(\mathcal{E})$ are Higgs nef.
    
    \item The line bundle $\det(E)$ is nef in the usual sense.
\end{enumerate}
\end{defi}

Note that if $\mathcal{E}$ is Higgs nef, then $\mathcal{O}_{\mathcal{G}r_k(\mathcal{E})}(1) \,=\, \det(\mathcal{Q}_{\mathcal{E},k})$ is nef in the usual sense for each $1\,\leq\, k \,\leq\, r-1$.

\begin{defi}[{\cite{BCO25}}]\label{defi1} 
Let $\mathcal{E} \,=\, (E,\,\theta)$ be a Higgs vector bundle of rank $r$ on a smooth  projective variety $X$. If $\rank(E) \,=\, 1$, then $\mathcal{E}$ is {\it Higgs ample} ($\text{H}$-ample in short) if $E$ is ample as a line bundle in the usual sense. If $\rank(E) \,\geq \,2$, then the Higgs ampleness ($\text{H}$-ampleness in short) of $\mathcal{E}$ is defined inductively as follows:
\begin{enumerate}
    \item For all $k$ with $1\,\leq\, k \,\leq\, r-1$, the universal Higgs quotient bundles $\mathcal{Q}_{\mathcal{E},k}$ on Higgs Grassmann schemes $\mathcal{G}r_k(\mathcal{E})$ are Higgs ample.
    
    \item The line bundle $\det(E)$ is ample in the usual sense.
\end{enumerate}
\end{defi}

\begin{xrem}\label{re1}
Definition \ref{defi1}  does not generalize the usual notion of ampleness of ordinary vector bundles  when the Higgs field $\theta$ is zero. According to Definition \ref{defi1}, if the Higgs field $\theta$ is zero, then \text{H}-ampleness of the Higgs bundle $\mathcal{E}\,=\,(E, \,\theta)$ implies that the universal quotient bundles $Q_{E,k}$ are ample in the usual sense. This is not the case as universal quotient bundles of a Grassmann variety is nef but not  ample; see \cite[Example 6.1.6]{L1}.  
\end{xrem}

To address the discrepancy mentioned in Remark \ref{re1}, we have introduced a new notion of Higgs ampleness.

\begin{defi}\label{defi2}
Let $\mathcal{E} \,=\, (E,\,\theta)$ be a Higgs vector bundle of rank $r$ on a smooth  projective variety $X$. We say that $\mathcal{E}$ is {\it Higgs ample} ($\text{H}$-ample in short) if it is H-nef and the following two conditions are satisfied :
    \begin{enumerate}
        \item For all $k$ with $1\,\leq\, k \,\leq\, r-1$, the line bundles $\mathcal{O}_{\mathcal{G}r_k(\mathcal{E})} (1)$  on Higgs Grassmann schemes $\mathcal{G}r_k(\mathcal{E})$ are ample in the usual sense.
        
        \item The line bundle $\det(E)$ is ample in usual sense.
    \end{enumerate}
\end{defi}

\begin{xrem}
If $E$ is an ample vector bundle, then the Higgs vector bundle $\mathcal{E}\,=\,(E,\,\theta)$ is Higgs ample for any Higgs field $\theta$ according to the Definition \ref{defi2}.  Indeed, if $E$ is ample, then for each $1\,\leq\, k \,\leq\, \rank(E)$, the exterior power $\bigwedge^kE$ are also ample. So for the Pl\"ucker embedding $\omega \,:\, \text{Gr}_k(E)\, \hookrightarrow \,\mathbb{P}(\wedge^kE)$, we have $\mathcal{O}_{\text{Gr}_k(E)}(1)\, =\, \omega^*\mathcal{O}_{\mathbb{P}(\bigwedge^k(E))}(1)$ to be ample. Thus for any Higgs field $\theta$ on $E$, the line bundle $\mathcal{O}_{\mathcal{G}r_k(\mathcal{E})}(1) \,=\, \mathcal{O}_{\text{Gr}_k(E)}(1)\vert_{\mathcal{G}r_k(\mathcal{E})}$  is ample. This implies that $\mathcal{E}\,=\,(E,\,\theta)$ is Higgs ample.
    \end{xrem}

\begin{xrem}
According to Definition \ref{defi2}, when the Higgs field $\theta$ is 0, the Higgs bundle $\mathcal{E} \,=\, (E,\,0)$ is Higgs ample if and only if $E$ is ample as a vector bundle in the usual sense. Indeed, as $\mathcal{G}r_k(\mathcal{E}) \,=\, \text{Gr}_k(E)$ for every $k$ when the Higgs field $\theta \,=\, 0.$ Also, in this case, $\mathcal{O}_{\mathcal{G}r_k(\mathcal{E})}(1)\, =\, \mathcal{O}_{\text{Gr}_k(E)}(1)$ for every $1\,\leq\, k \,\leq\, r-1$. Thus Definition \ref{defi2} is a natural generalization of the usual ampleness
of vector bundles.
\end{xrem}

\begin{xrem}
If a Higgs vector bundle $\mathcal{E}$ is $\text{H}$-ample (in the sense of \cite{BCO25}) as in Definition \ref{defi1}, then $\mathcal{E}$ is $\text{H}$-ample according to Definition \ref{defi2}. This is because  for any Higgs bundle $\mathcal{E}\,=\,(E,\,\theta),$ the determinant of its universal Higgs quotient satisfies the condition $\det(\mathcal{Q}_{\mathcal{E},k}) \,=\, \mathcal{O}_{\mathcal{G}r_k(\mathcal{E})}(1)$ for every $1\,\leq\, k\,\leq\, \rank(E) -1.$ However, the converse is not true, i.e., a Higgs vector bundle $\mathcal{E}\,=\,(E,\,\theta)$ may be Higgs ample according to Definition \ref{defi2} without being Higgs ample according to Definition \ref{defi1} (see Remark \ref{re1}).
\end{xrem}

\begin{xrem}
In the case of an ordinary vector bundle $E$, ampleness of $E$ is determined  by the ampleness of the universal rank-one quotient $\mathcal{O}_{\text{Gr}_1(E)}(1)\, =\, \mathcal{O}_{\mathbb{P}(E)}(1)$; equivalently, it suffices to check the ampleness of the tautological line bundle $\mathcal{O}_{\mathbb{P}(E)}(1)$ on the projective bundle $\mathbb{P}(E)$. In contrast to this, for Higgs bundles case, to check ampleness one must prove the ampleness of all the tautological line bundles $\mathcal{O}_{\mathcal{G}r_k(\mathcal{E})}(1)$ associated with the universal Higgs quotients $\mathcal{Q}_{\mathcal{E},k}$, of all ranks $1\,\leq \,k \, \leq\, r-1$.

Let $\mathcal{E} \,=\, (E,\,\theta)$ be a rank three nilpotent Higgs bundle on a smooth curve $C$, which is of the form $E\,=\,L_1\bigoplus L_2\bigoplus L_3$, where each $L_i$ is a line bundle and $\theta(L_1)\,\subset \,L_2\bigotimes \Omega^1_C$, $\theta(L_2)\,\subset\, L_3\bigotimes \Omega^1_C$ and $\theta(L_3)\,=\,0$.
Also let $\alpha_i\,=\,c_i(L_i)$, and let $F_k$ be the fibre of the map $\rho_k\,:\,\mathcal{G}r_k(\mathcal{E})\, \longrightarrow \,C$ and $c_1(\mathcal{Q}_{\mathcal{E},k})\,=\,\xi_k$. Assume that $\alpha_1,\alpha_2$ and $\alpha_3$ be integers such that $\alpha_1\,>\,0$, $\alpha_3\,<\,0$ and $\alpha_1+\alpha_2+\alpha_3\,>\,0$. 
The computations in Section 3.4 of \cite{BR06} show that $\mathcal{G}r_1(E)$ has two components not contained in the fibres which are $X_1$ and $X_2$ that are numerically equivalent to $2(\xi_1^2-(\alpha_2+\alpha_3)\xi_1\cdot F_1)$ and $\xi_1^2-(\alpha_2+\alpha_3)\xi_1\cdot F_1$ respectively. Now clearly $\deg(\xi_1\vert_{X_1})\,=\,2\alpha_1$
and $\deg(\xi_1\vert_{X_2})\,=\,\alpha_1$. Therefore, $\xi_1$ is ample on $\mathcal{G}r_1(E)$. Now $\mathcal{G}r_2(E)$ has one of the components $Y$ numerically equivalent to $\xi_2^2-(\alpha_1+\alpha_2)\xi_2\cdot F_2$. Also, $\deg(\xi_2\vert_Y)=\alpha_3 <0$ and hence $\xi_2\vert_Y$ is not ample. Hence $\xi_2$ is not ample.
\end{xrem}

\section{Main results}

Henceforth, we will use Definition \ref{defi2}. So a nef Higgs vector bundle $\mathcal{E} \,=\, (E,\, \theta)$ is Higgs ample if it satisfies the two conditions in Definition \ref{defi2}.

\begin{prop}\label{thm3.1}
 Let $\mathcal{E} \,= \,(E,\,\theta)$ be a Higgs vector bundle of rank $r$ on a smooth  projective variety $X$, and let $\phi \,:\, Y\,\longrightarrow\, X$ be a finite morphism of smooth projective varieties. Then the following two hold:
 \begin{enumerate}
     \item If $\mathcal{E}$ is Higgs ample on $X$, then its pullback $\phi^*\mathcal{E}$ on $Y$ is also Higgs ample.
     \item If $\phi$ is a finite surjective map and $\phi^*\mathcal{E}$ is Higgs ample, then $\mathcal{E}$ is also Higgs ample.
 \end{enumerate}
\end{prop}
 
 \begin{proof}
     Note that, by the universal property of Higgs Grassmann bundle, for each $k$ with $1\,\leq\, k \,\leq \, r-1$, there exists a finite map $\widetilde{\phi} \,:\, \mathcal{G}r_k(\phi^*\mathcal{E})\,\longrightarrow \,\mathcal{G}r_k(\mathcal{E})$ such that we have the following commutative fiber product diagram 
\begin{center}
 \begin{tikzcd} 
 \mathcal{G}r_k(\phi^*\mathcal{E}) \arrow[r, "\widetilde{\phi}"] \arrow[d, "\tilde{{\rho}_k}"]
& \mathcal{G}r_k(\mathcal{E}) \arrow[d,"\rho_k"]\\
Y\arrow[r, "\phi" ]
& X
\end{tikzcd}
\end{center}
and also for each $k$ with $1\,\leq\, k \,\leq \,r-1$, $$\widetilde{\phi}^*\mathcal{Q}_{\mathcal{E},k} \ =\ \mathcal{Q}_{{\phi}^*\mathcal{E},k}.$$
     This implies that $\det\bigl( \widetilde{\phi}^*\mathcal{Q}_{\mathcal{E},k} \bigr) \,=\, \det\bigl(\mathcal{Q}_{{\phi}^*\mathcal{E},k}\bigr),$ i.e., $ \widetilde{\phi}^*\mathcal{O}_{\mathcal{G}r_k(\mathcal{E})}(1)\, =\, \mathcal{O}_{\mathcal{G}r_k(\phi^*\mathcal{E})}(1)$ for each $1\,\leq\, k \,\leq\, r-1.$ This show that if $\mathcal{E}$ is Higgs ample, meaning each $\mathcal{O}_{\mathcal{G}r_k(\mathcal{E})}(1)$ is ample, then $\mathcal{O}_{\mathcal{G}r_k(\phi^*\mathcal{E})}(1)$ is also ample for each $k$ by \cite[Chapter 1]{L1}. This proves that $\phi^*\mathcal{E}$ is Higgs ample.

     Moreover, if $\phi$ is a finite surjective map and $\phi^*\mathcal{E}$ is Higgs ample so that $\mathcal{O}_{\mathcal{G}r_k(\phi^*\mathcal{E})}(1)$ is ample for each $k$, then $\mathcal{O}_{\mathcal{G}r_k(\mathcal{E})}(1)$ is ample for each $k$  by \cite[Chapter 1]{L1}, in other words, $\mathcal{E}$ is Higgs ample. This completes the proof.
 \end{proof}

\begin{prop}\label{thm3.2}
If $\mathcal{E}$ is an ample Higgs vector bundle of rank $r$ on a smooth projective variety $X$, then every Higgs quotient vector bundle of $\mathcal{E}$ is Higgs ample.
\end{prop}

\begin{proof}
    Let $\mathcal{F}$ be a Higgs quotient vector bundle of $\mathcal{E}$. Then, for each $k$ with $1\,\leq\, k \,\leq\, \rank(\mathcal{F})$, there exists a morphism $i\,:\,\mathcal{G}r_k(\mathcal{F}) \,\longrightarrow\, \mathcal{G}r_k(\mathcal{E})$ over $X$ such that $i^*\mathcal{Q}_{\mathcal{E},k} \,=\, \mathcal{Q}_{\mathcal{F},k}$. 

    Thus, we have $\det\bigl( i^*\mathcal{Q}_{\mathcal{E},k} \bigr) \,=\, i^*\mathcal{O}_{\mathcal{G}r_k(\mathcal{E})}(1) \,=\, \det\bigl( \mathcal{Q}_{\mathcal{F},k} \bigr) \,=\, \mathcal{O}_{\mathcal{G}r_k(\mathcal{F})}(1).$
    Now the result follows from \cite[Chapter 1]{L1}.
\end{proof}

\begin{xrem}
Let $\mathcal{E}$ be a Higgs vector bundle over a smooth projective variety $X$. Then $\mathcal{E}$ is Higgs ample if and only if the following two condition hold: 

\begin{itemize}
        \item The line bundle $\det(E)$ is ample.
        \item For every positive dimensional subvariety $Y$ of $X$ (including $X$ itself) with $i\,:\,Y \,\hookrightarrow \,X$ being the inclusion map, all the Higgs quotient bundles $\mathcal{Q}$ of $i^*\mathcal{E}$ are Higgs ample.
    \end{itemize}
    This follows from Proposition \ref{thm3.1} and Proposition \ref{thm3.2} applied to  $i^*\mathcal{E}$.
\end{xrem}

\begin{thm}\label{thm3.3}
     Let $X$ be a smooth projective variety, and let $h \,\in\, {\rm NS}(X)_{\mathbb{R}}$ be a fixed ample class on $X$. Then a Higgs vector bundle $\mathcal{E}\,=\,(E,\,\theta)$ of rank $r$ is Higgs ample if and only if the  following two conditions hold:

    \vspace{2mm}

    \begin{itemize}
        \item The line bundle $\det(E)$ is ample.
        \item  For any map $f\,:\,C\,\longrightarrow\, X$ from a smooth projective curve $C$ to $X$, there is a positive real number $\delta \,>\,0$ such that 
    $$\mu_{min}^{H}(f^*\mathcal{E})\,\ \geq\,\ \delta\bigl(C\cdot f^*h \bigr).$$
    Here $\delta$ is independent of $C$ and $f$.
    \end{itemize}
\end{thm}
    
\begin{proof}
    We first assume that the following two conditions hold: 
    \begin{itemize}
        \item The line bundle $\det(E)$ is ample.
        
        \item  For any map $f\,:\,C\,\longrightarrow\, X$ from a smooth projective curve $C$ to $X$, there is a positive real number $\delta \,>\,0$ (independent of $C$ and $f$) such that 
    $$\mu_{min}^{H}(f^*\mathcal{E})\,\ \geq\,\ \delta\bigl(C\cdot f^*h \bigr)\,\ \geq\,\ 0.$$
    \end{itemize}
    By \cite[Lemma 3.3]{BBG19} the Higgs bundle $\mathcal{E}$ is Higgs nef because $\mu_{min}^{H}(f^*\mathcal{E})\,\geq\, 0$. Denote $\xi_k \,=\, \mathcal{O}_{\mathcal{G}r_k(\mathcal{E})}(1).$ We will show that $\xi_k$ are ample for all $1\,\leq\, k \,\leq \,r-1.$ 
    
    Our claim is that $\xi_k-\delta\rho_k^*h$ is nef. For this it is enough to show that if $B$ is an irreducible curve in $\mathcal{G}r_k(\mathcal{E})$, then $$\bigl(\xi_k-\delta\rho_k^*h\bigr)\cdot B\,\ \geq\,\ 0.$$
    Note that if $B$ is contained in some fiber of $\rho_k,$ then $$\bigl(\xi_k-\delta\rho_k^*h\bigr)\cdot B\ =\ \xi_k \cdot B\ \geq\ 0,$$
    because $\xi_k$ is relatively ample with respect to the map $\rho_k.$ On the other hand, if $B$ is not contained in any fiber, then consider the normalization of $B$, which is denoted by $\widetilde{B}.$ We denote the following composition map by $\psi_k$ where $\widetilde{B} \, \longrightarrow\, B $ is the normalization map: 
    $$\psi_k \,:\, \widetilde{B}\, \longrightarrow\, B \,\hookrightarrow \,
    \mathcal{G}r_k(\mathcal{E}).$$ 

    Set $\eta_B \ =\ \rho_k\circ \psi_k.$
    By hypothesis we then have 
\begin{align}\label{seq4}
\mu_{min}^{H}(\eta_B^*\mathcal{E})\ \geq\ \delta\bigl(\widetilde{B}\cdot \eta_B^*h \bigr)\ \geq\ 0.
\end{align}
    Now $$\bigl(\xi_k-\delta\rho_k^*h\bigr)\cdot B\ =\ \xi_k\cdot B - \delta\bigl(\rho_k^*h \cdot B \bigr).$$

    Note that we have the following short exact sequence \begin{align*}
 0\, \longrightarrow\, \mathcal{S}_{\mathcal{E},r-k}\, \xrightarrow{\,\,\,\psi\,\,\,}\, \rho_k^*\mathcal{E} \, \xrightarrow{\,\,\,\eta\,\,\,}\,
 \mathcal{Q}_{\mathcal{E},k}\,\longrightarrow\, 0.
\end{align*}
Thus we have 
$$\eta_B^*\mathcal{E}\ =\ \psi_k^*\rho_k^* E\ \longrightarrow\ \psi_k^*\mathcal{Q}_{\mathcal{E},k} \ \longrightarrow\ 0.$$
This implies that $$\xi_k\cdot B\ \geq\ \mu\bigl( \psi_k^*\mathcal{Q}_{\mathcal{E},k}\bigr)\ \geq \ \mu^H_{\min}\bigl(\eta_B^*\mathcal{E}\bigr).$$
Consequently, using \eqref{seq4} we have
$$\bigl(\xi_k-\delta\rho_k^*h\bigr)\cdot B\ =\ \xi_k\cdot B - \delta\bigl(\rho_k^*h \cdot B \bigr) \ \geq\ \mu^H_{\min}\bigl(\eta_B^*\mathcal{E}\bigr) - \delta\bigl(\tilde{B}\cdot \eta_B^*h \bigr) \ \geq\ 0.$$

This completes the proof of our claim.

Next note that as $\xi_k$ is relatively ample, we have  $a\xi_k+\delta\rho_k^*h$ is ample for $0\,<\,a\, \ll \,1.$ 
Thus $$\bigl(\xi_k-\delta\rho_k^*h\bigr) + \bigl( a\xi_k+\delta\rho_k^*h\bigr)\ =\ (1+a)\xi_k$$ is ample for $0 \,<\, a \,\ll \,1.$ Hence $\xi_k$ is ample for each $k$, and $\det(E)$ is ample by hypothesis. This shows that $\mathcal{E}$ is Higgs ample.

Conversely, suppose $\mathcal{E}$ is Higgs ample. Therefore the line bundle $\det(E)$ is ample by definition. We just need to prove the following:
    
    For any map $f\,:\,C\,\longrightarrow\, X$ from a smooth projective curve $C$ to $X$, there is a positive real number $\delta \,>\,0$ ( independent of $C$ and $f$) such that 
    $$\mu_{min}^{H}(f^*\mathcal{E})\ \geq\ \delta\bigl(C\cdot f^*h \bigr).$$

    We consider the following two cases : 

    \begin{enumerate}
        \item {\bf Case 1}:\  Suppose that $f^*\mathcal{E}$ is a slope semistable Higgs vector bundle. By \cite[Corollary 1.4.11]{L1} there exists an $\epsilon \,>\, 0$ such that 
        $$\dfrac{\det(E)\cdot B}{B\cdot h}\ \geq\ \epsilon$$ for every irreducible curve $B\,\subseteq\, X$.
        Then $$\mu_{\min}^H\bigl(f^*\mathcal{E}\bigr)\, = \,\mu\bigl(f^*\mathcal{E}\bigr)\, =\, \dfrac{\det(E)\cdot\widetilde{C}}{r} \,\geq\, \dfrac{\epsilon}{r}\bigl(h\cdot \widetilde{C}\bigr) \,=\, \delta\bigl(C\cdot f^*h \bigr),$$
        where $\widetilde{C}$ is the image of $C$ in $X$ and $\delta \,=\, \dfrac{\epsilon}{r}.$

         \item {\bf Case 2}:\ Assume that $f^*\mathcal{E}$ is not Higgs semistable, and let \begin{align*}
  0 \,= \,\HN_0(\mathcal{E}) \,\subsetneq\, \HN_1(\mathcal{E}) \,\subsetneq\, \HN_2(\mathcal{E}) \,\subsetneq\, \cdots\,\subsetneq \,\HN_{l-1}(\mathcal{E}) \,\subsetneq\, \HN_l(\mathcal{E})\,=\, f^*\mathcal{E}
 \end{align*}
 be the Harder-Narasimhan filtration of $f^*\mathcal{E}.$ Let $s\,=\, \rank\bigl( f^*\mathcal{E}/\HN_{l-1}(\mathcal{E})\bigr)$. Then by the universal property of Higgs Grassmannian, there exists a lift 
 $f_s \,:\, C \,\longrightarrow \,\mathcal{G}r_s(\mathcal{E})$ of $f$ such that 
 $$\bigl(f^*\mathcal{E}/\HN_{l-1}(\mathcal{E})\bigr)\ =\ f_s^*\mathcal{Q}_{\mathcal{E},s}.$$
 
 As $\mathcal{O}_{\mathcal{G}r_s(\mathcal{E})}(1) \,=\, \det(\mathcal{Q}_{\mathcal{E},s})$ is ample, by \cite[Proposition 1.3.7]{L1} there exists a positive real number $\eta\,>\,0$ such that 
 $$\det(\mathcal{Q}_{\mathcal{E},s})-\eta\rho_s^*h$$ is ample, so that 
 \begin{align*}
     \mu_{\min}(f^*\mathcal{E}) \,=\, \dfrac{1}{s}\deg(f_s^*\mathcal{Q}_{\mathcal{E},s})\, =\, \dfrac{f_s^*\bigl(\det(\mathcal{Q}_{\mathcal{E},s})\bigr)\cdot C}{s} \,\geq\, \delta (C\cdot f^*h),
 \end{align*}
 where $\delta \,=\, \dfrac{\eta}{r-1}.$ 
 \end{enumerate}
This completes the proof.
\end{proof}

\begin{corl}[{Barton-Kleimann type criterion for \text{H}-ampleness}]\label{corl3.4}
Let $X$ be a smooth projective variety, and let $h \,\in \,{\rm NS}(X)_{\mathbb{R}}$ be a fixed ample class on $X$. Then a Higgs vector bundle $\mathcal{E}\,=\,(E,\,\theta)$ is Higgs ample if and only if the  following two conditions holds:

    \vspace{2mm}

    \begin{enumerate}
        \item The line bundle $\det(E)$ is ample.
        
        \item  There exists a $\delta\,\in\, \mathbb{R}_{>0}$ such that for every morphism $f\,:\,C\,\longrightarrow\, X$ from a smooth projective curve $C$, and for every Higgs quotient $\mathcal{Q}$ of $f^*\mathcal{E},$ the inequality
    $$\deg(\mathcal{Q})\,\ \geq\,\ \delta\bigl(C\cdot f^*h \bigr)$$ holds.
    \end{enumerate}
\end{corl}

    \begin{proof}
        If $\mathcal{E}$ is $\text{H}$-ample, then there exists a positive number $\delta' \,>\, 0$ such that 
        $$\deg(\mathcal{Q})\ =\ k\mu(\mathcal{Q})\ \geq\ k\mu_{min}(f^*\mathcal{E})\ \geq\ \delta(C\cdot f^*h),$$
        where $\rank(\mathcal{Q})\,=\,k$ and $\delta \,=\, k\delta' \,>\, 0$.
        
        Conversely, assume that $\mathcal{E} \,=\, (E,\,\theta)$ satisfies the two conditions (1) and (2). Now considering the Harder-Narasimhan filtration of $f^*\mathcal{E}$: 
        \begin{align*}
  0 \,=\, \HN_0(\mathcal{E}) \,\subsetneq\, \HN_1(\mathcal{E}) \,\subsetneq\, \HN_2(\mathcal{E})\, \subsetneq \,\cdots\,\subsetneq \,\HN_{l-1}(\mathcal{E}) \,\subsetneq\, \HN_l(\mathcal{E}) \,=\, f^*\mathcal{E}.
 \end{align*}
 We have $$\mu_{\min}^H(f^*\mathcal{E})\,=\, \dfrac{\deg(f^*\mathcal{E}/ \HN_{l-1}(\mathcal{E}))}{s} \,\geq\, \delta' \bigl(C\cdot f^*h\bigr),$$
 where $s\,=\,\rank(f^*\mathcal{E}/ \HN_{l-1}(\mathcal{E}))$ and $\delta'\,=\,\dfrac{\delta}{s}.$ Now the result follows from Theorem \ref{thm3.3}.
\end{proof}

\begin{corl}
    Let $\mathcal{E} \,=\, (E,\,\theta)$ and $\mathcal{F} \,=\, (F,\,\phi)$ be two Higgs vector bundles on a smooth projective variety $X$. Then
     if $\mathcal{E}$ is Higgs nef and $\mathcal{F}$ is Higgs ample, then $\mathcal{E} \otimes \mathcal{F}$ is Higgs ample.
\end{corl}
     
     \begin{proof}
First  note that $$\det(E\otimes F)\ =\ \det(E)^{\otimes \rank(F)}\otimes \det(F)^{\otimes \rank(E)}.$$ 
 As  $\mathcal{E}$ is Higgs nef, theline bundle $\det(E)$ is nef, and since  $\mathcal{F}$ is Higgs ample, $\det(F)$ is ample. This shows that $\det(E\otimes F)$ is ample using \cite[Chapter 1]{L1}.

 We fix an ample class $h$ on $X$. Now, for any smooth curve $C$ and any map $f\,:\,C\,\longrightarrow \,X$, we have, by \cite[Lemma 3.3]{BBG19}, $\mu_{\min}^H(f^*E)\,\geq \,0$ and there exists a $\delta\, >\, 0$  (independent of $C$ and $f$) such that 
    $$\mu_{min}^{H}(f^*\mathcal{E})\ \geq\ \delta\bigl(C\cdot f^*h \bigr).$$

    Note that $$\mu_{min}^{H}(f^*(\mathcal{E}\otimes\mathcal{F}))\ =\ \mu_{min}^H(f^*\mathcal{E})+\mu_{min}(f^*\mathcal{F})\ \geq\ \delta\bigl(C\cdot f^*h \bigr).$$
    Thus the results follows from Theorem \ref{thm3.3}.
    \end{proof}

\begin{corl}\label{thm3.6}
    Let $$0\,\longrightarrow\, \mathcal{E}_1\,\longrightarrow\, \mathcal{E} \,\longrightarrow\, \mathcal{E}_2 \,\longrightarrow\, 0$$ be a short exact sequence of Higgs vector bundles over a smooth projective variety $X$, where $\mathcal{E}_1\,=\,(E_1,\,\theta_1)$ and
    $\mathcal{E}_2\,=\,(E_2,\,\theta_2)$ are Higgs ample. Then $\mathcal{E}$ is Higgs ample.
\end{corl}
    
    \begin{proof}
    First note that $\det(E) \,= \,\det(E_1)\otimes \det(E_2)$ is ample. We fix an ample class $h$ on $X$. For any curve $C$ and any map $f\,:\,C\,\longrightarrow\, X,$ we choose a Higgs quotient $\mathcal{Q}$ of $f^*\mathcal{E}$. Then we can form the following diagram with exact rows and columns:
        \begin{center}
\begin{tikzcd} 
0 \arrow[r, " "] &f^*\mathcal{E}_1\arrow[r, " "] \arrow[d, " "] 
& \arrow[r, " "] f^*\mathcal{E} \arrow[r, " "] \arrow[d," "] 
& f^*\mathcal{E}_2 \arrow[r, "  "] \arrow[d," "] & 0  \\
0 \arrow[r, " "] & \mathcal{Q}_1  \arrow[r, " " ] \arrow[d, " "] & \mathcal{Q} \arrow[r, " "] \arrow[d," "]
&\mathcal{Q}_2 \arrow[r, " "] \arrow[d," "] & 0 \\
 & 0 & 0 & 0 
\end{tikzcd}
\end{center}
        Let $\mathcal{Q}_2'$ be $\mathcal{Q}_2$ modulo its torsion. Then by Theorem \ref{thm3.3} there exist $\delta_1\,>\,0$ and $\delta_2\,>\,0$ such that 
        $$\deg(\mathcal{Q}_1)\ \geq\ \delta_1\bigl(C\cdot f^*h\bigr), \hspace{3mm} \deg(\mathcal{Q}_2)\ \geq\ \deg(\mathcal{Q}_2')\ \geq\ \delta_2\bigl(C\cdot f^*h\bigr).$$ Let $\delta\,=\,\delta_1+\delta_2$, and thus we have $$\deg(\mathcal{Q})\ =\ \deg(\mathcal{Q}_1)+\deg(\mathcal{Q}_2)\ \geq\ \delta\bigl(C\cdot f^*h).$$
        Now the result follows from Theorem \ref{thm3.3}.
    \end{proof}

\begin{thm}\label{thm3.7}
Let $\mathcal{E}$ be a curve semistable Higgs vector bundle. If $c_1(E)$ is ample, then $\mathcal{E}$ is Higgs ample.
\end{thm}

\begin{proof}
Let $C$ be any smooth curve, and let $f\,:\,C\,\longrightarrow\, X$ be any morphism. As $\det(E)$ is ample, by \cite[Corollary 1.4.11]{L1} there exists an $\epsilon \,>\, 0$ such that 
        $$\dfrac{\det(E)\cdot B}{B\cdot h}\ \geq\ \epsilon$$ for every irreducible curve $B\,\subseteq \,X$.
    Since $\mathcal{E}$ is a curve semistable Higgs vector bundle, we have $$\mu_{\min}^H\bigl(f^*\mathcal{E}\bigr)\ =\ \mu\bigl(f^*\mathcal{E}\bigr)\ =\ \dfrac{\det(E)\cdot \widetilde{C}}{r}\ \geq\ \dfrac{\epsilon}{r}\bigl(h\cdot \widetilde{C}\bigr)\ =\ \delta\bigl(C\cdot f^*h \bigr),$$
        where $\widetilde{C}$ is the image of $C$ in $X$ and $\delta \,=\, \dfrac{\epsilon}{r}.$
        Again, using Theorem \ref{thm3.3} we conclude that $\mathcal{E}$ is Higgs ample.
        \end{proof}

\begin{thm}
A Higgs vector bundle $\mathcal{E}$ over a smooth projective curve $X$ is Higgs ample if and only if it has a positive degree and every Higgs quotient bundle of $\mathcal{E}$ has a positive degree.
\end{thm}

\begin{proof}
    Suppose that the Higgs bundle $\mathcal{E} \,=\, (E,\,\theta)$ over the smooth projective curve $X$ is Higgs ample. Then by definition, the line bundle $\det(E)$ is ample, and thus $\deg(E) \,=\, \deg(\det(E)) \,>\, 0.$ Similarly, if $\mathcal{E}\,\longrightarrow \,\mathcal{F} \,\longrightarrow\, 0$ is a Higgs quotient of $\mathcal{E}$, then $\mathcal{F}$ is also Higgs ample by Proposition \ref{thm3.2}. Therefore, we have $\deg(\mathcal{F}) \,=\, \deg(\det(F))\, > \,0,$ where $F$ is the underlying vector bundle of the Higgs bundle $\mathcal{F}.$

    Conversely, suppose that $\mathcal{E} \,=\, (E,\,\theta)$ is a Higgs bundle such that $\deg(E)\,>\,0$ and all of its Higgs quotient bundles have positive degree. Note that $\det(E)$ is always ample as $\deg(\det(E))\,=\,\deg(E) \,>\, 0.$ We consider the following two cases : 
    \begin{enumerate}
        \item {\bf Case 1 :}\  If $\mathcal{E}$ is slope Higgs semistable, then $\mathcal{E}$ is curve semistable by \cite[Lemma 3.3]{BR06}. Thus by Corollary \ref{thm3.7} we have $\mathcal{E}$ to be Higgs ample.

        \item {\bf Case 2 :}\ Suppose that $\mathcal{E}$ is not slope Higgs semistable, and let
         \begin{align*}
  0 \,= \,\HN_0(\mathcal{E}) \,\subsetneq\, \HN_1(\mathcal{E})\,\subsetneq \,\HN_2(\mathcal{E})\, \subsetneq\, \cdots\,\subsetneq\, \HN_{l-1}(\mathcal{E}) \,\subsetneq \,\HN_l(\mathcal{E})\,= \,\mathcal{E}
 \end{align*}
 be the Harder-Narasimhan filtration of $\mathcal{E}.$ Note that, there must be some index $i\,\in\,\{1,\,2,\,\cdots,\,l\}$ such that  $\HN_{i}(\mathcal{E})/ \HN_{i-1}(\mathcal{E})$ is not Higgs ample. Indeed, otherwise by Corollary \ref{thm3.6} we have $\mathcal{E}$ to be Higgs ample. Let $ \HN_{k}(\mathcal{E})/ \HN_{k-1}(\mathcal{E})$ be not Higgs ample. Then $ \HN_{k}(\mathcal{E})/ \HN_{k-1}(\mathcal{E})$ must have non-positive degree. Otherwise, by a similar argument as in Case 1, we have $ \HN_{k}(\mathcal{E})/ \HN_{k-1}(\mathcal{E})$ to be Higgs ample. By the property of the Harder-Narasimhan filtration, we have
 $$0\ \geq\ \mu(\HN_{k}(\mathcal{E})/ \HN_{k-1}(\mathcal{E})) \ \geq\ \mu(\HN_{l}(\mathcal{E})/ \HN_{l-1}(\mathcal{E})).$$ This gives a contradiction to the given hypothesis.
 \end{enumerate}
 Combining both the cases, we have the result.
\end{proof}

\section{Acknowledgement}
The authors are thankful to Krishna Hanumanthu for fruitful discussions.


\begin{thebibliography}{********}

\bibitem[BBG19]{BBG19}Indranil Biswas, Ugo Bruzzo and Sudarshan Gurjar, 
\emph{Higgs bundles and fundamental group schemes,}  Advances in Geometry, vol. 19, no. 3, pp. 381-388 (2019).

\bibitem[BCO25]{BCO25} Ugo Bruzzo,  Armando Capasso and  Beatriz Gra\~{n}a Otero, 
\emph{Positivity for Higgs Vector Bundles: Criteria and Applications,} Revista Matem\'{a}tica Complutense (2025), DOI : https://doi.org/10.1007/s13163-025-00551-7

\bibitem[BR06]{BR06} Ugo Bruzzo and D. Hern\'{a}ndez Ruip\'{e}rez,
\emph{Semistability vs. nefness for (Higgs) vector bundles,}
Differential Geometry and its Applications, Volume 24, Issue 4,  Pages 403-416, (2026).

\bibitem[CP91]{CP91}  Fr\'{e}d\'{e}ric Campana and Thomas Peternell,
\emph{Projective manifolds whose tangent bundles are numerically effective,}
Math Annalen, 289, pp 169-187, (1991).

\bibitem[H66]{H66} Robin Hartshorne,
\emph{Ample Vector Bundles,}
Publ. Math. IHES, 29, pp 63-94 (1966).

\bibitem[HL10]{HL10} D. Huybrechts and M. Lehn,
\emph{The Geometry of Moduli Spaces of Sheaves},
Second Edition, 2010. Cambridge University Press.

\bibitem[L04a]{L1} Robert Lazarsfeld,
\emph{Positivity in Algebraic Geometry,} Volume I, Springer, (2004).

\bibitem[L04b]{L2} Robert Lazarsfeld,
\emph{Positivity in Algebraic Geometry,} Volume II, Springer, (2004).

\end{thebibliography}
\end{document}